\documentclass[11pt, a4paper]{article}

\usepackage{pdfsync}
\usepackage{amsmath}\multlinegap=0pt
\usepackage[latin1]{inputenc}
\usepackage{amsthm}
\usepackage{amssymb}
\usepackage[francais,english]{babel}
\numberwithin{equation}{section}
\theoremstyle{plain}
\newtheorem{thm}{Theorem}[section]

\newtheorem{prop}[thm]{Proposition}
\newtheorem{lem}[thm]{Lemma}

\theoremstyle{definition}

\theoremstyle{remark}
\newtheorem{rem}[thm]{Remark}

\newcommand{\R}{\mathbb{R}}
\newcommand{\T}{\mathbb{T}}

\newcommand\N{{\mathbb N}}
\newcommand\Z{{\mathbb Z}}

\newcommand\pref[1]{(\ref{#1})}

\let \eps\varepsilon

\def\ds{\displaystyle}
\def\rg{\rangle} 
\def\lg{\langle}

\newcommand\M{{\cal M}}

\newcommand\dive{\mathrm{div}}

\def\<#1,#2>{\left<#1,#2\right>}
\let\bar\overline

\DeclareMathOperator{\BV}{BV}
\DeclareMathOperator{\ac}{ac}
\DeclareMathOperator{\rel}{rel}

\usepackage{color}

\title{Geodesics for a class of distances in the space of probability measures}

\author {P. ~Cardaliaguet, G.~Carlier, B.~Nazaret\thanks{\scriptsize CEREMADE, UMR CNRS 7534, Universit\'e Paris Dauphine, Pl. de Lattre de Tassigny, 75775 Paris Cedex 16, FRANCE
\texttt{cardaliaguet@ceremade.dauphine.fr, carlier@ceremade.dauphine.fr, nazaret@ceremade.dauphine.fr}.}}

\begin{document}
\maketitle

\begin{abstract}
In this paper, we study the characterization of geodesics for a class of distances between probability measures introduced by Dolbeault, Nazaret and Savar\'e. We first prove the existence of a potential function and then give necessary and sufficient optimality conditions that take the form of a coupled system of PDEs somehow similar to the Mean-Field-Games  system of Lasry and Lions. We also consider an equivalent formulation posed in a set of probability measures over curves. 
\end{abstract}


\section{Introduction}

Since the work of Benamou and Brenier \cite{bb} which showed that the squared $2$-Wasserstein distance between two probability measures $\rho_0$ and $\rho_1$ on $\R^d$ can be expressed as the infimum of the kinetic energy
\[\int_0^1 \int_{\R^d} \vert v(t,x)\vert^2 d\rho_t(x)dt\]
among solutions of the continuity equation with prescribed endpoints
\[\partial_t \rho_t+ \dive(\rho_t v)=0,  \; \rho(0,.)=\rho_0, \; \rho(1,.)=\rho_1\] 
 it is natural to view optimal transport theory from a dynamical perspective and to look for geodesics rather than just for transport maps. Similarly, probability measure valued curves governed by the continuity equation (with specific dependence of the velocity field $v$ on the mass $\rho$) play a crucial role in the theory of gradient flows in the Wassertsein space (see \cite{ags} and the references therein). 

\smallskip

In the recent paper \cite{dolnazsav}, Dolbeault, Nazaret and Savar\'e, introduced a new class of distances between probability measures through the introduction of some concave increasing nonnegative nonlinear mobility function $m$. Assuming that $\rho_0=\rho_0{\mathcal L}^d$ and $\rho_1=\rho_1{\mathcal L}^d$, the corresponding squared distance between $\rho_0$ and $\rho_1$ is (formally) given by   the infimum of 
\[\int_0^1 \int_{\R^d} \vert v(t,x)\vert^2 m(\rho(t,x))dxdt\]
subject to
\[\partial_t \rho+ \dive(m(\rho) v)=0,  \; \rho(0,.)=\rho_0, \; \rho(1,.)=\rho_1.\] 
From a Riemannian-like metric viewpoint, the formal interpretation of this mobility function is a conformal deformation factor of the Riemannian tensor on the tangent space of the space of probability measures at a given point $\rho$. From a modelling viewpoint, the concavity of $m$  is well suited to capture some congestion effects i.e. the fact that crowded zones of high densities result in higher values of the metric. Since McCann's pioneering work on displacement convexity \cite{McCann}, it is well known that displacement convex functionals play a distinguished role in the theory of gradient flows on the space of probability measures equipped with the $2$-Wasserstein distance. In \cite{clss}, Carrillo, Lisini, Slepcev and Savar\'e identified structural assumptions that guarantee the convexity of internal energy functionals along geodesics for the Dolbeault, Nazaret and Savar\'e distances with a general mobility function $m$ and this analysis will actually show useful for certain estimates in the present paper. 

 Let us mention that the monotonicity assumption on the mobility $m$ can be replaced by assuming that $m$ is compactly supported in an interval $(0,M)$, corresponding to a hard congestion effect, since here the density cannot assume values larger than $M$. Such a case has been treated in \cite{lm} and also considered in \cite{clss}. The main example that enters into this setting is $m(\rho)=\rho(1-\rho)$, and a study of geodesics for the associated distance can be found in the article of Brenier and Puel \cite{bp}, in the context of optimal multiphase transportation with a momentum constraint.

In this paper, we will focus on the case of a concave power mobility, $m(\rho)=\rho^\alpha$, $\alpha\in (0,1)$, and to avoid both compactness and boundary conditions issues, instead of working on a domain or $\R^d$, we will consider the case of the flat torus $\T^d:=\R^d/\Z^d$. Our goal is to characterize minimizing geodesics i.e. minimizers of
\[W_{\alpha}(\rho_0, \rho_1)^{2}:=
\inf\left\{\int_0^1 \int_{\T^d} \rho^{\alpha} \vert v\vert^2  \right\}\]
subject to the constraint
\[\partial_t \rho+\dive(\rho^{\alpha} v)=0, \;  \rho(0,.)=\rho_0, \; \rho(1,.)=\rho_1.\]
It is worth noting as in \cite{dolnazsav} that $W_{\alpha}(\rho_0, \rho_1)$ naturally interpolates between the $H^{-1}$ and the $2$-Wasserstein distance when $\alpha$ varies between $0$ and $1$ respectively. Just as in \cite{bb} taking $(\rho, w):=(\rho, \rho^{\alpha} v)$ as new variables, it is easy to see that this problem is a convex minimization problem, that is dual to some variational problem posed over some set of potentials $\phi$ (which naturally play the role of Lagrange multipliers associated to the constraint $\partial_t \rho+\dive(w)=0$).  Formally, the optimality conditions for $\rho$, $\phi$ obtained by convex duality read as the system
\[\left\{\begin{array}{llll}
&\partial_t \rho +\dive\Big(\frac{1}{2} \rho^\alpha \nabla \phi\Big)=0,\\
&\rho>0 \Rightarrow \partial_t \phi +\frac{\alpha}{4} \rho^{\alpha-1} \vert \nabla \phi \vert^2=0, \\
& \rho\geq 0, \; \partial_t \phi\leq 0, \; \rho(0,.)=\rho_0, \; \rho(1,.)=\rho_1.
\end{array}\right.\]
Let us remark that this system presents some similarities with the Mean-Field-Games  system of Lasry and Lions \cite{mfg2}, \cite{mfg3}. However, in the dual formulation the energy to be minimized is of the form  
\begin{equation}\label{intro:dual}
\kappa_\alpha \int_0^1 \int_{\T^d}  \Big(\frac{\vert \nabla \phi \vert^{\frac{2}{\alpha}}}{-\partial_t \phi} \Big)^{\frac{\alpha}{1-\alpha}}+ \int_{\T^d} \phi(0,.) \rho_0-\int_{\T^d} \phi(1,.) \rho_1
\end{equation}
for some positive constant $\kappa_\alpha$. Since this somehow nonstandard functional is not obviously coercive on some Sobolev space, the existence of a potential by the direct method of the calculus of variations  is not immediate at all.  In order to obtain estimates on minimizing sequences, a key ingredient is an estimate for some geodesic distances on $\T^d$ given by Lemma \ref{lem:estica}. In particular, we have to assume here that $\alpha>1-\frac{2}{d}$, an assumption that recurrently appears in previous works on this distances.  Secondly, the potential we obtain is merely $\BV$ in time so some extra work has to be done to derive and justify rigorously  the system of optimality conditions.

\smallskip

This result and its proof suggest a different approach to the problem, by somehow lifting  the geodesics problem to a variational problem at the level of measures on curves, which is reminiscent to the work of Carlier-Jimenez-Santambrogio \cite{cjs} (see also \cite{bcs}). More precisely, given $\eta$  a periodic measure supported by a suitable  set $\Gamma$ of  curves and which connects $\rho_0$ to $\rho_1$ (in the sense that the image of $\eta$ by the evalution maps at initial and terminal time are respectively $\rho_0$ and $\rho_1$), define the measure $\sigma_\eta$ on $[0,1]\times \T^d$  through
$$
\int_0^1 \int_{\T^d}f(s,x)d\sigma_\eta(s,x)= \int_\Gamma\int_0^1 f(s,\gamma(s))\left|\dot \gamma(s)\right|^{2/(2-\alpha)}\ dsd\eta(\gamma)\;,
$$
for any continuous and periodic in space  function $f$. We will investigate the precise links between the geodesics problem above and the minimization of
\[\eta \mapsto \int_0^1\int_{\T^d} \left(\sigma_\eta(t,x)\right)^{2-\alpha}\ dx\ dt.\]


\smallskip

The paper is organized as follows. In section \ref{existadj}, when $\alpha>1-\frac{2}{d}$, we establish the existence and uniqueness of a potential function (or adjoint state) and give necessary and sufficient optimality conditions characterizing geodesics in section \ref{optimalcondi}. Section \ref{sec:PbDualCourbe} is devoted to the  equivalent reformulation as a variational problem posed on the set of probability measures over curves.

\section{Existence of an adjoint state}\label{existadj}

In this section, we introduce (a suitably relaxed version of) the minimization problem \eqref{intro:dual}, dual to the problem defining the distance between two probability densities. The existence of a solution strongly relies of the existence of a minimizing sequence satisfying some a priori estimates, the proof of which is postponed to subsection~\ref{aprioriest}.

\subsection{Duality}

Let $\alpha\in (0,1)$, for $(\rho, w)\in \R\times \R^d$, let us define
\[H(\rho,w):=\left\{\begin{array}{llll}
&\frac{\vert w\vert^2}{\rho^\alpha} &\mbox{ if } \rho>0\\
&  0 &\mbox{ if } (\rho,w)=(0,0)\\
&  +\infty &\mbox{ otherwise }
\end{array}\right.\]
and for $(a,b)\in \R\times \R^d$,
\[L(a,b):=\left\{\begin{array}{llll}
& \kappa_\alpha  \Big( \frac{\vert b\vert^{\frac{2}{\alpha}}}{-a} \Big)^{\frac{\alpha}{1-\alpha}}&\mbox{ if } a<0\\
&  0 &\mbox{ if } (a,b)=(0,0)\\
&  +\infty &\mbox{ otherwise }
\end{array}\right.\]
where
$$
\kappa_\alpha:=(1-\alpha) \frac{ \alpha^{\frac{\alpha}{1-\alpha}}}{4^{\frac{1}{1-\alpha}}}.
$$
By direct computation, one checks that $L$ and $H$ are convex lsc and conjugates: $L=H^*$, $H=L^*$  and that for $(\rho,w)\in \R_+\times \R^d$, one has
\[\partial H(\rho, w)=\left\{\begin{array}{lll}
& \Big(-\alpha \frac{\vert w \vert^2}{\rho^{\alpha+1}}, \frac{2w}{\rho^\alpha}\Big) &\mbox{ if } \rho>0\\
&  \R_-\times \{0\} &\mbox{ if } (\rho,w)=(0,0).
\end{array}\right.\]
Recall that we work in the space-periodic framework, setting $\T^d:=\R^d/\Z^d$ and $Q:=[-1/2,1/2]^d$, this means that we identify spaces of  functions on $\T^d$ to spaces of  $Q$-periodic functions on $\R^d$.  Let us then consider 
\begin{equation}\label{primal}
\inf_{\phi\in C^1([0,1]\times \T^d)} J(\phi):=\int_0^1 \int_Q L(\partial_t \phi, \nabla \phi) dx dt+\int_Q \phi(0,.) \rho_0-\int_Q \phi(1,.) \rho_1
\end{equation}
where $\rho_0$, $\rho_1$ are two given probability measures on $Q$.  We shall see that in some sense to be made more precise, the variational problem \pref{primal} admits as dual the Dolbeault-Nazaret-Savar\'e problem in \cite{dolnazsav}
\begin{equation}\label{dns}
-W_\alpha(\rho_0, \rho_1)^2:=\sup_{(\rho, w)} -\int_0^1 \int_Q H(\rho^{\ac}, w) dx dt
\end{equation}
where $\rho^{\ac}$ stands for the absolutely continuous part of the measure $\rho$ with respect to the Lebesgue measure and the supremum is performed among pairs of measures $(\rho, w)\in \M_+([0,1]\times \T^d)\times \M([0,1]\times \T^d)^d$ such that $w$ is absolutely continuous with respect to the Lebesgue measure and $(\rho, w)$ is a  weak solution of  the continuity equation
\begin{equation}\label{cont}
\partial_t \rho +\dive(w)=0, \; \rho(0,.)=\rho_0, \; \rho(1,.)=\rho_1
\end{equation}
i.e. satisfies
\begin{equation}\label{contweak}
\int_0^1 \int_Q  \partial_t \phi  d\rho + \int_0^1 \int_Q  \nabla \phi \cdot w= \int_Q \phi(1,.) \rho_1-\int_Q \phi(0,.) \rho_0
\end{equation}
for every $\phi\in C^1([0,1]\times \T^d)$. The fact that \pref{dns} coincides with the problem studied in \cite{dolnazsav} has to be a little bit further explained: first remark that if $(\rho, w)\in \M_+([0,1]\times \T^d)\times \M([0,1]\times \T^d)^d$ solves \pref{contweak} then the time marginal of $\rho$ is the Lebesgue measure. By disintegration we can thus write $d\rho=\rho_t \otimes dt$. Note now that $w\in L^{1}((0,1)\times \T^d)$, which implies (see for instance Lemma  4.1 in \cite{dolnazsav})  that $t\mapsto \rho_t$ is continuous for the weak $*$ topology of $\M$. Finally $\rho^{\ac}(t,x)=\rho_t^{\ac}(x)$ so that the functional in \pref{dns} can be rewritten as $\int_0^1 \int_Q H(\rho_t^{\ac}(x), w(t,x) )dx dt$ which is exactly the functional to be minimized in \cite{dolnazsav}.

\smallskip

Assuming $\rho_0$ and $\rho_1$ belong to $L^1(\T^d)$, we shall also need to suitably relax \pref{primal} by considering
\begin{equation}\label{primalr}
\inf_{\phi\in K} J^{\rel}(\phi):=\int_0^1 \int_Q L(\partial_t \phi^{\ac}, \nabla \phi) dx dt+\int_Q \phi(0,.) \rho_0-\int_Q \phi(1,.) \rho_1
\end{equation}
where 
\[K:=\{\phi\in \BV\cap L^{\infty} \; : \; \partial_t \phi \leq 0, \; \nabla \phi\in L^1\}\]
and $\partial_t \phi^{\ac}$ denotes the absolutely continuous part of the measure $\partial_t \phi $ with respect to the Lebesgue measure. Since elements of $K$ are bounded and monotone nonincreasing with respect to time, the second and third term in $J^{\rel}(\phi)$ are well-defined by monotone convergence i.e.  $\phi(0,.)$ and $\phi(1,.)$ are intended as $\phi(0,.):=\phi(0^+,.)=\sup_{t\in(0,1)} \phi(t,.)$, and  $\phi(1,.)=\phi(1^-,.)=\inf_{t\in(0,1)} \phi(t,.)$. For further use, let us also remark that for $\phi\in K$, by Beppo-Levi's monotone convergence theorem, one has
\begin{equation}\label{btsci1}
\int_Q \phi(0,.) \rho_0=\sup_{\delta\in (0,1)} \frac{1}{\delta} \int_{[0,\delta]\times Q} \phi(t,x) \rho_0(x)dx dt, \; 
\end{equation}
\begin{equation}\label{btsci2}
\int_Q \phi(1,.) \rho_1=\inf_{\delta\in (0,1)} \frac{1}{\delta} \int_{[1-\delta, 1]\times Q} \phi(t,x) \rho_1(x)dx dt.
\end{equation}

\begin{prop}\label{dualitydns} Let $\rho_0$ and $\rho_1$ be two probability measures on $\T^d$ such that $W_\alpha(\rho_0, \rho_1)<+\infty$, then 
\[-W_\alpha(\rho_0, \rho_1)^2=\inf \pref{primal}\]
and the infimum in \pref{dns} is attained.
If furthermore $\rho_0$ and $\rho_1$ belong to $L^1(\T^d)$ then, in addition, $\inf \pref{primal} =\inf \pref{primalr}.$
\end{prop}

\begin{rem} If we consider a general concave and nonnegative mobility fonction $m$, the analogue of the problem \eqref{primal} writes
$$
\inf\left[-\int_0^1\int_Q m^\star\left(-\frac{\partial_t\phi}{\left|\nabla\phi\right|^2}\right)\left|\nabla\phi\right|^2 dxdt\right],
$$
where $m^\star$ stands for the concave Legendre transform of $m$, defined by
$$
m^\star(\sigma) = \inf_{\rho\in\R}\left(\sigma\cdot\rho-m(\rho)\right).
$$
The corresponding duality theorem could be proved for such a general $m$. However, the main difficulty here is to get the existence of an adjoint state: getting estimates as in Proposition~\ref{prop:constrphin} for a general mobility function $m$ is an open problem. 
\end{rem}

\begin{proof}
Let us rewrite \pref{primal} as 
\[\inf_{\phi\in C^1([0,1]\times \T^d)} \{F(D \phi)+G(\phi)\}\]
where $D$ : $C^1([0,1] \times \T^d)\to C([0,1]\times \T^d)\times C([0,1]\times \T^d)^d$ is defined by $D\phi:=(\partial_t \phi, \nabla \phi)$, 
\[F(a,b):=\int_0^1 \int_Q L(a(t,x), b(t,x)) dx dt,
\]
$\forall (a,b)\in C([0,1]\times \T^d)\times C([0,1]\times \T^d)^d$ and 
\[G(\phi):=\int_Q \phi(0,.) \rho_0-\int_Q \phi(1,.) \rho_1.\]
Note that, if $\phi\in C^1([0,1]\times \T^d)$ is such that $\partial_t \phi\leq -\delta_0<0$, then $F$ is continuous (for the uniform topology) at $D\phi$. Since $G$ is continuous and $W_\alpha(\rho_0, \rho_1)<+\infty$, Fenchel-Rockafellar's duality theorem thus implies that 
\begin{equation}\label{fenchel}
\inf \pref{primal}=\max_{(\rho, w)\in \M\times \M^d} \{-F^*(\rho, w)-G^*(-D^*(\rho, w))\}
\end{equation}
where $\M:=\M([0,1]\times \T^d)$ is of course identified to the topological dual of $C([0,1]\times \T^d)$. By direct computation, we have 
\[G^*(-D^*(\rho, w))=\left\{\begin{array}{lll}
& 0 &\mbox{ if  $(\rho,w)$ solves \pref{cont}}\\
&  +\infty &\mbox{ otherwise}.
\end{array}\right.\]
Thanks to theorem 5 in \cite{rockafellar2}, and the fact that $L^*=H$, one has
\[F^*(\rho, w)=\int_0^1 \int_Q H(\rho^{\ac}, w^{\ac})+ \int_0^1 \int_Q H_\infty\Big (\frac{d \rho^s} {d\theta}, \frac{ d w^s} {d\theta}\Big) d\theta\]
where $(\rho^{\ac}, w^{\ac})$ and $(\rho^s, w^s)$ denote respectively the absolutely continuous part and singular part of $(\rho,w)$, $\theta$ is any measure with respect to which  $(\rho^s, w^s)$ is absolutely continuous (for instance $\rho^s +\vert w^s\vert$) and $H_\infty$ is the recession function of $H$:
\[H_\infty(\rho,w)=\sup_{\lambda>0} \frac{1}{\lambda} H(\lambda \rho, \lambda w)  =\left\{\begin{array}{lll}
& 0 &\mbox{ if  $\rho\geq 0$ and $w=0$}\\
&  +\infty &\mbox{ otherwise}.
\end{array}\right.\]
Replacing in \pref{fenchel}, we thus deduce that $\inf\pref{primal}= -W_\alpha(\rho_0, \rho_1)^2$. 

Next we assume that $\rho_0$ and $\rho_1$ belong to $L^1(\T^d)$ and $W_\alpha(\rho_0, \rho_1)<+\infty$. Then, still by the Fenchel-Rockafellar's duality theorem, the infimum in \pref{dns} is attained. 
Let us finally prove that $\inf \pref{primal} =\inf \pref{primalr}$. The fact that $\inf \pref{primal}\geq \inf \pref{primalr}$ is obvious. Let $\phi\in K$ (extended by $\phi(0,.)$ for $t\leq 0$ and by $\phi(1,.)$ for $t\geq 1$), then let $\phi^\eps:=\eta^\eps \star \phi$ where $\eta^\eps(t,x):=\eps^{-d-1} \alpha(\eps^{-1} t)\beta(\eps^{-1} x)$ with $\alpha \in C_c^{\infty}((-1/2,1/2))$, $\alpha\geq 0$, $\int_{-1/2}^{1/2} \alpha=1$, $\alpha$ even, $\beta \in C_c^{\infty}((-1/2,1/2)^d)$, $\beta\geq 0$, $\int_Q \beta=1$, $\beta$ even. Moreover we set $\tilde \phi^\eps(t,x)= \phi^\eps(\eps+(1-2\eps)t,x)$. 
To prove the remaining inequality it is enough to prove that
\begin{equation}\label{approx1}
J^{\rel}(\phi)\geq \limsup_{\eps\to 0^+} J(\tilde \phi^\eps).
\end{equation}  
To see this, we assume of course that $J^{\rel}(\phi)<+\infty$ and first remark that $\eta^\eps\star \partial_t \phi^{\ac} \geq \partial_t \phi^\eps$. Using the fact that $L$ is convex and nondecreasing in its first argument we thus get
\[\begin{split}
 \int_0^1 \int_Q L(\partial_t \tilde  \phi^\eps, \nabla\tilde  \phi^\eps) \leq \frac{1}{(1-\eps)^{\frac{\alpha}{1-\alpha}+1}}  \int_\eps^{1-2\eps} \int_Q L(\eta^\eps\star \partial_t \phi^{\ac}, \eta^\eps \star \nabla \phi)\\
 \leq  \frac{1}{(1-2\eps)^{\frac{1}{1-\alpha}}}
  \int_\eps^{1-\eps} \int_Q \eta^\eps \star L(\partial_t \phi^{\ac}, \nabla \phi)\to  \int_0^1 \int_Q L(\partial_t \phi^{\ac}, \nabla \phi) \mbox{ as $\eps\to 0^+$}.
\end{split}\]
We next observe that thanks to the monotonicity of $\phi$ in time, setting $\beta^\eps:=\eps^{-d} \beta(\eps^{-1}.)$, we have
\begin{equation}\label{convolin1}
\int_Q \tilde \phi^\eps(0,.) \rho_0 \leq \int_Q  \phi(0,.) \beta^\eps\star \rho_0\to \int_Q \phi(0,.) \rho_0 \mbox{ as $\eps\to 0^+$}
\end{equation}
and similarly 
\begin{equation}\label{convolin2}
\int_Q \tilde \phi^\eps(1,.) \rho_1 \geq \int_Q  \phi(1,.) \beta^\eps\star \rho_1\to \int_Q \phi(1,.) \rho_1 \mbox{ as $\eps\to 0^+$}
\end{equation}
so that \pref{approx1} holds. 
\end{proof}

The crucial step to establish the existence of an optimal potential  in \pref{primalr} is to find a minimizing sequence for \pref{primal} that satisfies the estimate
\[\Vert \phi_n\Vert_{L^{\infty}}+ \Vert \partial_t \phi_n\Vert_{L^1}+ \Vert \nabla \phi_n \Vert_{L^2} \leq C.\]
The proof of this fact is postponed to subsection \ref{aprioriest}. These bounds  enable us to prove the following result.

\begin{thm}\label{existdualdns}
If $\alpha>1-2/d$ and $\rho_0$ and $\rho_1$ belong to $L^1(\T^d)$. Then \pref{primalr} admits a unique solution $\phi$ up to an additive constant. Moreover $\nabla \phi \in L^2$. 
\end{thm}

\begin{rem}\label{rem:LienInt} Using the homogeneities in the functional $J^{\rel}$, one easily checks that, if $\phi$ is the unique minimizer, then 
$$
\frac{2-\alpha}{1-\alpha} \int_0^1\int_Q L(\partial_t \phi^{\ac}, \nabla \phi) = \int_Q \phi(1,\cdot)\rho_1-\int_Q \phi(0,\cdot)\rho_0\;.
$$
\end{rem} 

\begin{proof} According to proposition \ref{prop:constrphin} below, there is a 
a minimizing sequence $\phi_n$ of \pref{primal} (hence also of \pref{primalr}) and a constant $C>0$, such that for every $n$, $\partial_t \phi_n\leq -1/n$, and 
\begin{equation}\label{estimatespierre}
\Vert \phi_n\Vert_{L^{\infty}}+ \Vert \partial_t \phi_n\Vert_{L^1}+ \Vert \nabla \phi_n \Vert_{L^2} \leq C.
\end{equation}
Taking a subsequence if necessary and using the monotonicity of $\phi_n$ with respect to time, we may assume that there is a $\phi \in K$ such that $\nabla \phi\in L^2$ and
\begin{eqnarray}
&\phi_n \to \phi \mbox{ strongly in $L^1$ and weakly $*$ in $L^{\infty}$},\label{pierre1}\\
&\partial_t \phi_n \to \partial_t \phi  \mbox{ weakly $*$ in $\M$},\label{pierre2}\\
&\nabla \phi_n \to \nabla \phi  \mbox{ weakly in $L^2$}.\label{pierre3}
\end{eqnarray}
To prove that $\phi$ solves \pref{primal}, we first deduce from  theorem 5 in \cite{rockafellar2}, that the functional 
\[(\mu,\nu)\in \M\times \M^d \mapsto \left\{\begin{array}{llll}
& \ds{ \int_0^1\int_Q L(\mu^{\ac}, \nu) } &\mbox{ if  $\mu\leq 0$ and $\nu\in L^1$} \\
&  +\infty &\mbox{ otherwise }
\end{array}\right.\]
is the convex conjugate of 
\[(\theta,v)\in C([0,1]\times \T^d)\times C([0,1]\times \T^d)^d \mapsto \int_0^1\int_Q H(\theta(t,x), v(t,x))dx dt\ .\]
It is therefore lsc for the weak $*$ topology of $\M$ and therefore
\[\int_0^1 \int_Q L(\partial_t \phi^{\ac}, \nabla \phi)  \leq \liminf_n \int_0^1 \int_Q L(\partial_t \phi_n, \nabla \phi_n) dx dt .\]
We finally deduce from \pref{btsci1}, \pref{btsci2} that the second term in $J^{\rel}$ is lsc for the weak $*$ topology of $L^{\infty}$, thanks to \pref{pierre1}, we thus get
\[\int_Q \phi(0,.) \rho_0-\int_Q \phi(1,.) \rho_1 \leq \liminf_n \Big( \int_Q \phi_n(0,.) \rho_0-\int_Q \phi_n(1,.) \rho_1 \Big)\]
so that $\phi$ solves \pref{primalr}. 

\smallskip

To prove the uniqueness claim, let us first assume that $\rho_0\neq \rho_1$. We observe that if $(a,b)$ and $(a',b')$ are two distinct points in $\R_-\times \R^d$ at which $L$ is finite, then $L$ is strictly convex on $[(a,b), (a',b')]$ unless $b=b'=0$. Assume now  that   $\phi$ and $\psi$ are two solutions of \pref{primalr}. By the previous observation, we have, for almost all $(t,x)$, $$
{\rm either}\;  (\partial_t\phi^{\ac}(t,x), \nabla \phi(t,x)) = (\partial_t\psi^{\ac}(t,x),\nabla \psi(t,x))
$$
$$
{\rm or }\; \partial_t\phi^{\ac}(t,x)\neq \partial_t\psi^{\ac}(t,x)\;{\rm and}\; 
\nabla \phi(t,x)=\nabla \psi(t,x)=0\;.
$$
In particular $\nabla \phi=\nabla \psi$ a.e., so that there is a measurable map $\xi(t)$ with $\psi(t,x)=\phi(t,x)+\xi(t)$. Note that $\xi$ is BV, because $\phi$ and $\psi$ are BV.  Next we note that, for a.e. $t$ such that $\partial_t\xi^{\ac}(t)\neq 0$, we have $\partial_t\phi^{\ac}(t,x)\neq \partial_t\psi^{\ac}(t,x)$ and therefore $\nabla \phi(t,x)=\nabla \psi(t,x)=0$. Let us assume for a while that  the set $E$  of $t\in (0,1)$ for which $\partial_t\xi^{\ac}(t)\neq 0$ has a positive measure. Then, if $t$ is a density point of $E$, $\phi(t,x)=c$ a.e. for some constant $c$. Since $\phi(0,x)\geq \phi(t,x)=c\geq \phi(1,x)$ a.e., we get 
$$
\int_Q \phi(0,.) \rho_0\geq c \geq \int_Q \phi(1,.) \rho_1.
$$
Therefore
$$
 W_\alpha(\rho_0, \rho_1)^2= - 
\int_0^1 \int_Q L(\partial_t \phi^{\ac}, \nabla \phi) dx dt-\int_Q \phi(0,.) \rho_0+\int_Q \phi(1,.) \rho_1\leq 0
$$
which is impossible because $\rho_0\neq \rho_1$ and $W_\alpha$ is a distance. So $\partial_t\xi^{\ac}=0$. Next we prove that the singular measure $\partial_t \xi$ is zero. Let us decompose $\partial_t \xi$ into its positive and negative part $\partial_t \xi^+$ and $\partial_t \xi^-$ (i.e., $\partial_t \xi=
\partial_t \xi^+-\partial_t \xi^-$, $\partial_t \xi^+, \partial_t \xi^-\geq0$) 
and  let us set $\phi_1(t,x)=\phi(t,x)+ \partial_t \xi^+([0,t])$. Recall that $\partial_t \xi^+$ is concentrated on a Borel set $E\subset [0,1]$ such that $\partial_t\xi^-(E)=0$. We claim that $\phi_1$ is nonincreasing in time. Indeed, for any Borel set $A\times B \subset [0,1]\times Q$, 
$$
\partial_t \phi_1 (A\times B)= \partial_t \psi((A\cap E)\times B)+ \partial_t \phi((A\cap E^c)\times B)\leq 0\;.
$$
Note moreover that $\phi_1(0,\cdot)=\phi(0,\cdot)$, $\partial_t \phi_1^{\ac}=\partial_t \phi^{\ac}$,
$\nabla \phi_1=\nabla \phi$,  $\phi_1(1,\cdot)= \phi(1,\cdot)+ \partial_t \xi^+([0,1])$ so that 
$$
J(\phi_1)= \int_0^1 \int_Q L(\partial_t \phi^{\ac}, \nabla \phi) dx dt+\int_Q \phi(0,.) \rho_0-\int_Q \phi(1,.) \rho_1- \partial_t \xi^+([0,1])\;.
$$
Since $\phi$ is a minimizer, this implies that $\partial_t \xi^+([0,1])=0$, so that $\partial_t \xi^+=0$. Arguing in the same way with $\psi_1=\psi+\partial_t \xi^-([0,t])$, one gets that $\partial_t \xi^-=0$. In conclusion, $\xi$ is constant. 

If $\rho_0=\rho_1$, then $\phi=0$ is optimal. If $\psi$ is another optimal solution,  one can show as in the first part of the proof that $\psi(t,x)=\xi(t)$, where $\xi$ is a nonincreasing map. Computing the criterium for $\xi$ shows that $\xi(0)=\xi(1)$, so that again $\xi$ is constant. 
\end{proof}

A consequence of the proof of Theorem \ref{existdualdns} is the following technical remark, needed below: 

\begin{lem}\label{lem:LimPhin} Under the assumption of Theorem \ref{existdualdns}, let  $(\phi_n)$ be a minimizing sequence of \pref{primal} such that, for any $n$, $\partial_t \phi_n\leq -1/n$ and 
$$
\Vert \phi_n\Vert_{L^{\infty}}+ \Vert \partial_t \phi_n\Vert_{L^1}+ \Vert \nabla \phi_n \Vert_{L^2} \leq C\;,
$$
and let $\phi$ be a limit of $(\phi_n)$, in the sense that 
\begin{eqnarray*}
&\phi_n \to \phi \mbox{ strongly in $L^1$ and weakly $*$ in $L^{\infty}$},\\
&\partial_t \phi_n \to \partial_t \phi  \mbox{ weakly $*$ in $\M$},\\
&\nabla \phi_n \to \nabla \phi  \mbox{ weakly in $L^2$}.
\end{eqnarray*}
Then  $\phi$ is optimal for  \pref{primal} and
\begin{equation}\label{lim1}
\int_0^1 \int_Q L(\partial_t \phi^{\ac}, \nabla \phi)  =  \lim_n \int_0^1 \int_Q L(\partial_t \phi_n, \nabla \phi_n) dx dt \;,
\end{equation}
\begin{equation}\label{lim2}
\int_Q \phi(0,.) \rho_0-\int_Q \phi(1,.) \rho_1 = \lim_n \Big( \int_Q \phi_n(0,.) \rho_0-\int_Q \phi_n(1,.) \rho_1 \Big)\;.
\end{equation}
Moreover, if we set
\[ A_n=  
\frac{|\nabla \phi_n|^{\frac{2}{\alpha}}}{(-\partial_t\phi_n(t,x))} \; {\rm and } \; 
A= \frac{|\nabla \phi|^{\frac{2}{\alpha}}}{(-\partial_t\phi^{\ac}(t,x))}\]
then  $(A_n^{\frac{\alpha}{2-\alpha}})$ converges to $A^{\frac{\alpha}{2-\alpha}}$ in $L^{\frac{2-\alpha}{1-\alpha}}$. 
\end{lem}

\begin{proof} Equalities (\ref{lim1}), (\ref{lim2}) are straightforward consequences of the proof of Theorem \ref{existdualdns}. 
In view of (\ref{lim1}), $(A_n^{\frac{\alpha}{2-\alpha}})$ is bounded in $L^{\frac{2-\alpha}{1-\alpha}}$. Therefore a subsequence, still denoted $(A_n^{\frac{\alpha}{2-\alpha}})$, converges weakly to some $\tilde A\in L^{\frac{2-\alpha}{1-\alpha}}$. Let now $f:[0,1]\times \R^d\to \R$ be a continuous, positive and periodic map. Applying the argument of the proof of Theorem  \ref{existdualdns} to the convex functional
$$
J_f(\psi):= \int_0^1\int_Q \frac{|\nabla \psi|^{\frac{2}{2-\alpha}}}{(-\partial_t\psi(t,x))^{\frac{\alpha}{2-\alpha}}}f(x,t)\ dx\ dt\;,
$$
we get that 
\begin{eqnarray*}
\int_0^1\int_Q  \tilde A f & = & \liminf_n \int_0^1\int_Q A_n^{\frac{\alpha}{2-\alpha}} f\\
& = & \liminf_n J_f(\phi_n) \geq J_f(\phi) = \int_0^1\int_Q A^{\frac{\alpha}{2-\alpha}} f
\end{eqnarray*}
Therefore $\tilde A\geq A^{\frac{\alpha}{2-\alpha}}$. Furthermore, in view of (\ref{lim1}), 
$$
\int_0^1\int_Q A^{\frac{\alpha}{1-\alpha}}= \lim_n \int_0^1\int_Q (A_n^{\frac{\alpha}{2-\alpha}})^{\frac{2-\alpha}{1-\alpha}}\geq   \int_0^1\int_Q \tilde A^{\frac{2-\alpha}{1-\alpha}}
$$
so that $\tilde A= A^{\frac{\alpha}{2-\alpha}}$ and $(A_n^{\frac{\alpha}{2-\alpha}})$ strongly converges to $A^{\frac{\alpha}{2-\alpha}}$ in
$L^{\frac{2-\alpha}{1-\alpha}}$.
\end{proof}

Here is an elementary property of minimizers. 

\begin{prop}\label{prop:esssup} If $\phi$ is optimal for \pref{primalr}, then $t\mapsto {\rm ess-sup}_x \phi(t,x)$ and $t\to {\rm ess-inf}_x \phi(t,x)$
are constant. 
\end{prop}

\begin{proof} Let $M= {\rm ess-sup}_x \phi(1,x)$ and let  us consider $\tilde \phi= \inf\{M, \phi\}$. Then one easily checks that 
$$
\int_0^1\int_Q L(\partial_t \tilde \phi^{\ac}, \nabla \tilde \phi)\leq \int_0^1\int_Q L(\partial_t  \phi^{\ac}, \nabla \phi)
$$
while
$$
\int_Q  \tilde \phi(0,\cdot)\rho_0\leq \int_Q  \phi(0,\cdot) \rho_0\; {\rm and}\; 
\int_Q  \tilde \phi(1,\cdot)\rho_1= \int_Q  \phi(1,\cdot) \rho_1.
$$
Therefore $\tilde \phi$ is also optimal for \pref{primalr}, so that by uniqueness $\tilde \phi=\phi$. In particular, 
${\rm ess-sup}_x \phi(t,x)\leq M={\rm ess-sup}_x \phi(1,x)$ for any $t$. Since  $\phi$ is nonincreasing in time, we can conclude that 
$t\to {\rm ess-sup}_x \phi(t,x)$ is constant. The other assertion is proved similarly considering $\sup\{m, \phi\}$ with $m={\rm ess-inf}_x \phi(0,x)$.
\end{proof}

\subsection{Construction of a minimizing sequence for \pref{primal}}\label{aprioriest}

We now prove that there exists a minimizing sequence for \pref{primal} which satisfies the estimates needed in the proof of theorem \ref{existdualdns}: 

\begin{prop}\label{prop:constrphin} If $\alpha>1-2/d$, then there is a minimizing sequence $(\phi_n)$ for \pref{primal} and a constant $C>0$ such that 
$$
\Vert \phi_n\Vert_{L^{\infty}}+ \Vert \partial_t \phi_n\Vert_{L^1}+ \Vert \nabla \phi_n \Vert_{L^2} \leq C.
$$
\end{prop}

The proof relies on an integral estimate for geodesic distance. Let $\alpha\in (0,1)$ and $a:[0,1]\times \R^d\to \R$ be a smooth, positive and $\Z^d-$periodic map.  For $0\leq s<t\leq 1$ and any $x,y\in Q^2$, we consider the minimization problem
\begin{equation}\label{defca}
c_a((s,x),(t,y))=\inf_{\gamma} \int_s^t (a(\tau,\gamma(\tau)))^{\frac{\alpha}{2-\alpha}}\left|\dot\gamma(\tau)\right|^{\frac{2}{2-\alpha}}d\tau
\end{equation}
where the infimum is taken over the set   of 
all $W^{1,\frac{2}{2-\alpha}}$ maps $\gamma:[s,t]\to \R^d$ such that $\gamma(s)=x$ and $\gamma(t)=y$. It will be convenient for later use
to introduce the set $\Gamma((s,x),(t,y))$ of  continuous maps $\gamma:[s,t]\to \R^d$ such that $\gamma(s)=x$ and $\gamma(t)=y$ and to extend
the infimum in (\ref{defca}) to $\Gamma((s,x),(t,y))$ by setting
$$
 \int_s^t (a(\tau,\gamma(\tau)))^{\frac{\alpha}{2-\alpha}}\left|\dot\gamma(\tau)\right|^{\frac{2}{2-\alpha}}d\tau=+\infty
 \quad {\rm if }\; \gamma\notin W^{1,\frac{2}{2-\alpha}}.
 $$
 For simplicity, we denote respectively by $c_a(x,y)$ and $\Gamma(x,y)$ the quantity $c_a((0,x),(1,y))$ and the set $\Gamma((0,x),(1,y))$.

\begin{lem}\label{lem:estica} If $\alpha>1-2/d$, then, for any $\sigma\in (0, \frac{2-d(1-\alpha)}{(2-\alpha)})$, 
\begin{equation}\label{crucialest}
c_a((s,x),(t,y))\leq C \left[\int_s^t \! \int_{Q} (a(s,y))^{\frac{\alpha}{1-\alpha}}dyds\right]^{\frac{1-\alpha}{2-\alpha}} \frac{|x-y|^\sigma}{(t-s)^{\frac{1}{2-\alpha}}}
\end{equation}
where $C=C(d,\alpha,\sigma)$. 
\end{lem}

\begin{proof}  We start by proving the result for $s=0$, $t=1$. The general case is obtained by scaling. 
Let $x,y\in Q$.
Since $\alpha>1-2/d$ we can fix $\theta>0$ sufficiently large such that 
$$
\alpha > \frac{1+d\theta-2\theta}{d\theta}\;.
$$
Throughout the proof $C$ denotes a constant depending on $d$, $\alpha$ and  $\theta$, which may change from line to line. 

Let $\Pi$ be the set of Borel probability measures on $\Gamma(x,y)$ endowed with the $C^0$-distance. We first note that
$$
c_a(x,y)=\inf_{\pi\in \Pi} \int_{\Gamma(x,y)}\int_0^1 (a(s,\gamma(s)))^{\frac{\alpha}{2-\alpha}}\left|\dot\gamma(s)\right|^{\frac{2}{2-\alpha}}dsd\pi(\gamma)
$$
For any $\delta\in(0,1/3)$, $r\in [0,1]$ and $\sigma\in S^d$ ($S^d$ being the unit sphere of $\R^d$), we set
$$
\gamma^\delta_{r,\sigma}(s)= \left\{\begin{array}{ll}
x+r\sigma s^\theta & {\rm if }\; s\in [0,\delta]\\
x+r\sigma \delta^\theta + \frac{(y-x)(s-\delta)}{1-2\delta} & {\rm if }\; s\in [\delta,1-\delta]\\
y+r\sigma (1-s)^\theta & {\rm if }\; s\in [1-\delta,1]
\end{array}\right.
$$
and we define the probability measure $\pi^\delta\in \Pi$ by
$$
\int_{\Gamma(x,y)}\Phi(\gamma) d\pi^\delta(\gamma) = \int_0^1 \int_{S^d} \Phi(\gamma^\delta_{r,\sigma}) \frac{d}{|S^d|}r^{d-1}drd\sigma
$$
for any Borel measurable nonnegative map $\Phi:  \Gamma(x,y)\to \R$. 
Then 
$$
\begin{array}{rl}
c_a(x,y)\; \leq & \ds{ \int_0^1\!\!\!\int_0^1\!\!\!\int_{S^d}  (a(s,\gamma^\delta_{r,\sigma}(s)))^{\frac{\alpha}{2-\alpha}}\left|\dot\gamma^\delta_{r,\sigma}(s)\right|^{\frac{2}{2-\alpha}}
\frac{d}{|S^d|}r^{d-1} d\sigma dr ds  }\\
\leq & \ds{ I_1+I_2+I_3  }
\end{array}
$$
where 
$$
I_1= \int_0^{\delta}\!\!\!\int_0^1\!\!\!\int_{S^d}  (a(s,x+r\sigma s^\theta ))^{\frac{\alpha}{2-\alpha}}\left| r \theta s^{\theta-1} \right|^{\frac{2}{2-\alpha}} \frac{d}{|S^d|}r^{d-1} d\sigma dr ds \;,
$$
\begin{eqnarray*}
I_2 & = & \int_{\delta}^{1-\delta}\!\!\!\!\int_0^1\!\!\!\int_{S^d}  \left[a(s,x+r\sigma \delta^\theta + \frac{(y-x)(s-\delta)}{1-2\delta})\right]^{\frac{\alpha}{2-\alpha}}\times\\
& \qquad & \times\left|\frac{(x-y)}{1-2\delta}\right|^{\frac{2}{2-\alpha}}\frac{d}{|S^d|}r^{d-1} 
d\sigma dr  ds 
\end{eqnarray*}
and 
$$
I_3=\int_{1-\delta}^1\!\int_0^1\!\!\!\int_{S^d}  (a(s,y+r\sigma (1-s)^\theta))^{\frac{\alpha}{2-\alpha}}\left| r \theta (1-s)^{\theta-1} \right|^{\frac{2}{2-\alpha}} \frac{d}{|S^d|}r^{d-1} d\sigma dr  ds .
$$
Let us first estimate $I_2$. 
If we set $x_s= x+ (y-x)(3s-1)$ for $s\in [1/3,2/3]$, we have, by change of variables and by H\"older inequality, 
$$
\begin{array}{rl}
I_2\; = & \ds{  C\left|x-y\right|^{\frac{2}{2-\alpha}} \delta^{-d\theta} \int_{\delta}^{1-\delta} \int_{B(x_s, \delta^\theta)}
(a(s,z) )^{\frac{\alpha}{2-\alpha}} \ dz\ ds  }\\
\leq & \ds{  C\left|x-y\right|^{\frac{2}{2-\alpha}}\delta^{-d\theta} \left[\int_0^1\int_{{\mathcal C}(x,y)} (a(s,y))^{\frac{\alpha}{1-\alpha}}dyds\right]^{\frac{1-\alpha}{2-\alpha}} \delta^{\frac{\theta d}{2-\alpha}} }\\
\leq & \ds{  C\left|x-y\right|^{\frac{2}{2-\alpha}}\delta^{-\frac{d\theta(1-\alpha)}{2-\alpha}} \left[\int_0^1\int_{Q} (a(s,y))^{\frac{\alpha}{1-\alpha}}dyds\right]^{\frac{1-\alpha}{2-\alpha}} }
\end{array}
$$
where ${\mathcal C}(x,y)=\bigcup_{s\in [\delta,1-\delta]} B(x_s,\delta^\theta)\subset [-M,M]$ for some $M\in \N^*$ independent of $\delta$, $r$ and $\sigma$ (note then that the last inequality comes from the periodicity of $a$).

\noindent 
We now estimate $I_1$. By change of variable and H\"older inequality, we have: 
$$
\begin{array}{rl}
I_1 \; = & \ds{  C \int_0^{\delta}\int_{B(x, s^\theta)}   (a(s,z  ))^{\frac{\alpha}{2-\alpha}} \left|\frac{z-x}{s^\theta}\right|^{\frac{2}{2-\alpha}}  s^{-d\theta+2\frac{\theta-1}{2-\alpha}} dz  ds }\\
 = & \ds{  C \int_0^{\delta}\int_{B(x, s^\theta)}   (a(s,z  ))^{\frac{\alpha}{2-\alpha}} \left|z-x\right|^{\frac{2}{2-\alpha}}  s^{-d\theta-\frac{2}{2-\alpha}}  dz  ds }\\
\leq & \ds{ C \left[\int_0^\delta \int_{B(x,\delta^\theta)} (a(s,y))^{\frac{\alpha}{1-\alpha}}dyds\right]^{\frac{(1-\alpha)}{(2-\alpha)}}}\times\\
& \qquad {\ds \times\left[ \int_0^\delta\int_{B(x, s^\theta)} |z-x|^2  s^{-d\theta(2-\alpha)-2}  dz  ds \right]^{\frac{1}{(2-\alpha)}} }\\
\end{array}
$$
where the last term is given by
$$
\begin{array}{rl}
\ds{  \int_0^\delta \int_{B(x, s^\theta)} |z-x|^2 s^{-d\theta(2-\alpha)-2} dz  ds  }
\; = & 
\ds{  C\int_0^\delta\int_0^{s^\theta} r^{1+d}  s^{-d\theta(2-\alpha)-2}  dz  ds  }\\
=& 
\ds{  C\int_0^\delta s^{\theta(2+d) -d\theta(2-\alpha)-2 }  ds  }\\
= & \ds{  C \delta^{ \theta(2+d) -d\theta(2-\alpha)-1} }
\end{array}
$$
Note that $\theta(2+d) -d\theta(2-\alpha)-1>0$ thanks to the choice of $\theta$. The term $I_3$ can be estimated in the same way.Therefore 
\begin{eqnarray*}
c_a(x,y) & \leq & C \left[\int_0^1 \int_{Q} (a(s,y))^{\frac{\alpha}{1-\alpha}}dyds\right]^{\frac{(1-\alpha)}{(2-\alpha)}}\times\\
& & \qquad \times\left[ \delta^{ \frac{\theta(2+d) -d\theta(2-\alpha)-1}{(2-\alpha)}} + \left|x-y\right|^{\frac{2}{(2-\alpha)}}\delta^{-\frac{d\theta(1-\alpha)}{(2-\alpha)}} \right]
\end{eqnarray*}
Choosing $\ds \delta= \frac{|x-y|^{\frac{2}{2\theta-1}}}{4(\text{diam}(Q))^{\frac{2}{2\theta-1}}}\in (0,1/3)$, we get
$$
c_a(x,y)\; \leq \; C \left[\int_0^1 \int_{Q} (a(s,y))^{\frac{\alpha}{1-\alpha}}dyds\right]^{\frac{(1-\alpha)}{(2-\alpha)}}
\left|x-y\right|^{\frac{2}{2-\alpha}\left(1-\frac{d\theta(1-\alpha)}{2\theta-1}\right)}
$$
where the exponent
$$
\sigma=\sigma(\theta)=\frac{2}{2-\alpha}\left(1-\frac{d\theta(1-\alpha)}{2\theta-1}\right)
$$
is positive from the choice of $\theta$ and that its limit as $\theta\to+\infty$ is $\frac{2-d(1-\alpha)}{2-\alpha}$.

\medskip
It remains to check the general case. Let us notice that 
$$
c_a((s,x),(t,y))=(t-s)^{\frac{\alpha}{2-\alpha}}c_{\tilde a}(x,y)
$$
where $\tilde a(\tau,y)= a(s+\tau(t-s),y)$. Therefore 
$$
c_a((s,x),(t,y))\leq (t-s)^{\frac{\alpha}{2-\alpha}} C
 \left[\int_0^1 \int_{Q} (\tilde a(\tau,y))^{\frac{\alpha}{1-\alpha}}dyd\tau \right]^{\frac{(1-\alpha)}{(2-\alpha)}}
|x-y|^\sigma
$$
where 
$$
 \left[\int_0^1 \int_{Q} (\tilde a(\tau,y))^{\frac{\alpha}{1-\alpha}}dyd\tau \right]^{\frac{(1-\alpha)}{(2-\alpha)}}
= 
(t-s)^{-\frac{1-\alpha}{2-\alpha}} \left[\int_s^t \int_{Q} (\tilde a(\tau,y))^{\frac{\alpha}{1-\alpha}}dyd\tau \right]^{\frac{(1-\alpha)}{(2-\alpha)}}
$$
This gives the result. 
\end{proof}

\begin{proof}[Proof of proposition \ref{prop:constrphin}]
Let $\phi_n$ be a minimizing sequence for \pref{primal}. Without loss of generality we can assume that 
$$
\partial_t \phi_n\leq -\frac{1}{n}, \qquad \min_{x\in Q} \phi_n(0,x)=0\qquad \forall n\in \N\;.
$$
Let us set 
$$
A_n(t,x)= \frac{|\nabla\phi_n(t,x)|^{2/\alpha}}{(-\partial_t \phi_n(t,x))}\;.
$$
Then, if we set
$$
d_\alpha= \frac{(2-\alpha)\alpha^{\frac{\alpha}{2-\alpha}}}{2^{\frac{2}{2-\alpha}}},
$$
we have that, for any curve $\gamma:[0,1]\to \R^d$,
$$
\begin{array}{l}
\ds{ \frac{d}{dt} \left(\phi_n(t,\gamma(t)) -d_\alpha\int_0^t (A_n(s,\gamma(s)))^{\frac{\alpha}{2-\alpha}}\left|\dot \gamma(s)\right|^{\frac{2}{2-\alpha}} ds\right)  }\\  
\qquad = \; \ds{  \partial_t \phi_n+ \lg \nabla\phi_n, \dot \gamma\rg -d_\alpha A_n^{\frac{\alpha}{2-\alpha}}\left|\dot \gamma\right|^{\frac{2}{2-\alpha}}   }   \\
\qquad = \; \ds{  \partial_t \phi_n+  \lg \nabla\phi_n, \dot \gamma\rg 
-d_\alpha\frac{|\nabla\phi_n|^{\frac{2}{2-\alpha}}}{(-\partial_t \phi_n)^{\frac{\alpha}{2-\alpha}} }\left|\dot \gamma \right|^{\frac{2}{2-\alpha}}   }\\
\qquad \leq \; \ds{   \sup_{b\in \R^d} \left( \partial_t \phi_n+  \lg \nabla\phi_n, b\rg 
-d_\alpha\frac{|\nabla\phi_n|^{\frac{2}{2-\alpha}}}{(-\partial_t \phi_n)^{\frac{\alpha}{2-\alpha}} }\left|b\right|^{\frac{2}{2-\alpha}}  \right)
\ \leq \ 0\;. }
\end{array}
$$
Therefore, for any $x,y\in Q$, 
$$
\phi_n(1,y)\leq \inf_{\gamma} \left( \phi_n(0,\gamma(0))+d_\alpha\int_0^1 (A_n(s,\gamma(s)))^{\frac{\alpha}{2-\alpha}}\left|\dot \gamma(s)\right|^{\frac{2}{2-\alpha}} ds\right)
$$
 where the infimum is taken over absolutely continuous curves such that $\gamma(1)=y$. In view of Lemma \ref{lem:estica}, for any $\sigma\in (0, \frac{2-d(1-\alpha)}{(2-\alpha)})$, we have 
$$
\phi_n(1,y)\leq \inf_{x\in Q} \left( \phi_n(0,x)+ C \left[\int_0^1\int_{Q} (A_n(s,z))^{\frac{\alpha}{1-\alpha}}dzds\right]^{\frac{1-\alpha}{2-\alpha}} |x-y|^\sigma\right)
$$
for some constant $C=C(d,\alpha,\sigma)$. Since $\min_{x\in Q}  \phi_n(0,x)=0$, we finally get
$$
\phi_n(1,y)\leq C(d,\alpha,\sigma)\left[\int_0^1\int_{Q} (A_n(s,z))^{\frac{\alpha}{1-\alpha}}dzds\right]^{\frac{1-\alpha}{2-\alpha}}\;.
$$

Let now $\xi_n:\R\to \R$ be a smooth function such that $0<\xi_n'\leq 1$, $\xi_n(\tau)=\tau$ for $\tau\geq 0$ and $\xi_n\geq -1/n$. From now on we replace $\phi_n$ by $\xi_n\circ \phi_n$: the new sequence $(\phi_n)$ is still minimizing and
\begin{equation}\label{estiphi1}
-\frac{1}{n}\leq \phi_n(1,y)\leq M_n:=C(d,\alpha,\sigma)\left[\int_0^1\int_{Q} (A_n(s,z))^{\frac{\alpha}{1-\alpha}}dzds\right]^{\frac{1-\alpha}{2-\alpha}}.
\end{equation}

Next we modify $\phi_n(0,\cdot)$. Let $\zeta_n:\R\to \R$ be such that $0<\zeta_n' \leq 1$ with $\zeta_n(\tau)=\tau$ for $\tau\leq M_n$ and $\zeta_n\leq 2M_n$. Replacing $\phi_n$ by   $\zeta_n\circ \phi_n$, we get again a minimizing sequence such that 
$$
-\frac{1}{n}\leq \phi_n(s,y)\leq 2M_n\qquad \forall (s,y)\in [0,1]\times Q\;.
$$
It remains to prove that $M_n$ is bounded: since $\phi_n$ is minimizing, we have, for $n$ sufficiently large
\begin{eqnarray*} 
-W_\alpha^2(\rho_0,\rho_1) +1& \!\geq & 
\!\!\kappa_\alpha  \int_0^1\int_{Q}\! \left(\frac{|\nabla \phi_n|^{\frac{2}{\alpha}}}{(-\partial_t \phi_n)}\right)^{\frac{\alpha}{1-\alpha}}
+\int_Q \!\phi_n(0)d\rho_0 -\int_Q \!\phi_n(1)d\rho_1 \\
& \!\geq &  \!\!\kappa_\alpha  \int_0^1\int_{Q} \left(A_n \right)^{\frac{\alpha}{1-\alpha}}
-C\left[\int_0^1\int_{Q} (A_n)^{\frac{\alpha}{1-\alpha}}\right]^{\frac{1-\alpha}{2-\alpha}} ,
\end{eqnarray*}
thanks to (\ref{estiphi1}), $\phi_n(0,\cdot)\geq 0$ and the fact $\rho_1$ is a probability measure. Since $\frac{1-\alpha}{2-\alpha}<1$, we obtain that 
$\int_0^1\int_{Q} (A_n)^{\frac{\alpha}{1-\alpha}}$ is bounded, so that $M_n$ is also bounded. Replacing finally $\phi_n$ by $\phi_n+1/n$ gives that 
$$
0\leq \phi_n \leq C.
$$
In order to conlude, let us now show that 
$$
 \Vert \partial_t \phi_n\Vert_{L^1}+ \Vert \nabla \phi_n \Vert_{L^2} \leq C.
 $$
Indeed, by H\"{o}lder inequality, we have
$$
\int_0^1\int_Q |\nabla\phi_n|^2 \leq
\left[ \int_0^1\int_Q \left(\frac{|\nabla\phi_n(t,x)|^{\frac{2}{\alpha}}}{(-\partial_t \phi_n(t,x))}\right)^{\frac{\alpha}{1-\alpha}}\right]^{1-\alpha}
\left[ \int_0^1\int_Q |\partial_t\phi_n| \right]^\alpha
$$
where the right-hand side is bounded since
$$
\int_0^1\int_Q |\partial_t\phi_n| = \int_0^1\int_Q (-\partial_t\phi_n) = \int_Q \phi_n(0)d\rho_0-\int_Q\phi_n(1)d\rho_1 \leq C.
$$
\end{proof}

\section{Optimality conditions}\label{optimalcondi}

Our aim now is to use the duality between \pref{primalr} and \pref{dns} to write necessary and sufficient optimality conditions in the form of a (sort of) system of PDEs. To achieve this goal we have to be able to multiply $\rho$ by $\partial_t \phi^{\ac}$ and $w$ by $\nabla \phi$, since a priori  $\partial_t\phi^{\ac}$ is only $L^1$ we shall need a uniform  in time $L^{\infty}$ estimate on $\rho_t$ when $\rho_0$ and $\rho_1$ are $L^\infty$. This estimate will follow from a generalized displacement convexity argument of Carrillo, Lisini, Savar\'e and Slepcev \cite{clss}.  This estimate will also be  useful to treat the term
\[w\cdot \nabla \phi=\frac{w}{\rho^{\alpha/2}} \rho^{\alpha/2} \nabla \phi\]
if the functional in \pref{dns} is finite, one has $\frac{w}{\rho^{\alpha/2}} \in L^2$ and  $\rho^{\alpha/2} \nabla \phi \in L^2$ whenever $\nabla \phi\in L^2$ which ensures the summability of $w\cdot \nabla \phi$.

\begin{thm}\label{carrillo}
If $\rho_0$ and $\rho_1$ belong to $L^{\infty}(\T^d)$, then $W_{\alpha}(\rho_0, \rho_1)$ is finite and the infimum in \pref{dns} is achieved by a pair $(\rho, w)$, disintegrating $\rho$ as $d\rho=\rho_t \otimes dt$ we also have that $\rho_t \in L^{\infty}$ for every $t$ with
\begin{equation}\label{linfiniestim}
\Vert \rho_t \Vert_{L^{\infty}} \leq \max(\Vert \rho_0\Vert_{L^{\infty}}, \Vert \rho_1\Vert_{L^{\infty}})
\end{equation}
and $w\in L^2$.

\noindent
Furthermore, if $d=1$, and if there exists $C>0$ such that
$$
\rho_0,\rho_1\,\geq\, C \mbox{ a.e. on } \T^d,
$$
then $\rho_t\geq C$ a.e. on $\T^d$, for every $t$.
\end{thm}

\begin{proof}
The fact that $W_{\alpha}(\rho_0, \rho_1)$ is finite follows from Corollary 5.25 in \cite{dolnazsav}, the fact that the infimum is achieved then follows  for instance from  proposition \ref{dualitydns}. The estimate \pref{linfiniestim} is obtained by proving the convexity of $\rho\mapsto \int_Q \rho(x)^p dx$ along geodesics with respect to $W_\alpha$ for large $p$, and letting $p$ goes to $+\infty$. This has been done in \cite{clss}, in the case of an open bounded convex set $\Omega$ in $\R^d$, and is a consequence from the fact that the functional defined above generates a $C^0$-metric contraction gradient flow in the space of probability measures on $\Omega$ endowed with $W_\alpha$ (in the sense of \cite{ags}), given by the solution of the following porous medium equation on $\Omega$ with zero flux condition on the boundary
$$
\partial_t\rho_t = \Delta \left(\rho_t^{\alpha+p-1}\right) \mbox{ in } (0,+\infty)\times\Omega,
$$
for all $p$ such that
$$
p\geq 2-\left(1+\frac{1}{d}\right)\alpha \mbox{ and } p\not=1.
$$
Using the translation invariance of the equation, the reader may check that all the arguments of \cite{clss} can be directly adapted to the case of the flat torus.
To conclude the proof  of the first part, let us notice that, since $(\rho, w)$ has finite energy, $\frac{\vert w\vert^2}{\rho^{\alpha}} \in L^1$ so that $w\in L^2$.

When $d=1$, the displacement convexity results of \cite{clss} (again adapted to the case of the flat torus) imply that $t\mapsto \int_Q U(\rho_t(x))dx$ is convex for any convex $U$, we thus have, taking arbitrary negative powers,
\[\max_{t\in[0,1]} \Vert \rho_t^{-1}\Vert_{L^s}^s\leq \max (\Vert \rho_0^{-1} \Vert_{L^{\infty}}, \Vert \rho_1^{-1}\Vert_{L^{\infty}})\]
 for every $s>1$ and the desired claim is obtained by letting $s\to \infty$. 
\end{proof}

\begin{lem}\label{ippineg}
Assume that $\alpha>1-2/d$ and that $\rho_0$, $\rho_1$ belong to $L^{\infty}(\T^d)$.  Let $(\rho, w)\in L^{\infty}\times L^2$ be a weak solution of \pref{cont} and $\phi\in K$ be such that $\nabla \phi\in L^2$ then 
\begin{equation}\label{ipppositif}
\int_{0}^1 \int_Q (\rho \partial_t \phi^{\ac} + w \cdot \nabla \phi)dxdt +\int_Q (\phi(0,.) \rho_0-\phi(1,.) \rho_1)dx\geq 0
\end{equation}
\end{lem}

\begin{proof}
We regularize $\phi$ exactly as in the proof of proposition \ref{dualitydns} and thus set $\phi^\eps:=\eta_\eps \star \phi$, since $\partial_t \phi^{\ac} \geq \partial_t \phi$, thanks to \pref{cont} , defining $\rho^\eps:=\eta^\eps \star \rho$  (having extended $\rho$ by $0$ outside $[0,1]$ which is consistent with the fact that we have extended $\partial_t \phi$ in the same way), thanks to the periodicity and $\partial_t \phi \leq \partial_t \phi^{\ac}$,  we then have
\[\begin{split}
0&=\int_{0}^1 \int_Q (\rho \partial_t \phi^{\eps} + w \cdot \nabla \phi^\eps) +\int_Q (\phi^\eps(0,.) \rho_0-\phi^\eps(1,.) \rho_1)\\
&=\int_{0}^1 \int_Q (\rho^\eps \partial_t \phi + w \cdot \nabla \phi^\eps) +\int_Q (\phi^\eps(0,.) \rho_0-\phi^\eps(1,.) \rho_1)\\
&\leq \int_{0}^1 \int_Q (\rho^\eps \partial_t \phi^{\ac} + w \cdot \nabla \phi^\eps) +\int_Q (\phi^\eps(0,.) \rho_0-\phi^\eps(1,.) \rho_1)
\end{split}\]
Since $w$ and $\nabla \phi$ are in $L^2$, $\int_{0}^1 \int_Q w \cdot \nabla \phi^\eps$ converges to $\int_{0}^1 \int_Q w \cdot \nabla \phi$. Moreover, since $\rho^\eps$ is uniformly bounded and converges to $\rho$ in $L^1$---hence a.e. up to a subsequence---, Lebesgue's  dominated convergence theorem implies that $\int_{0}^1 \int_Q \rho^\eps \partial_t \phi^{\ac}$ converges to $\int_{0}^1 \int_Q \rho \partial_t \phi^{\ac}$. Recalling \pref{convolin1}, \pref{convolin2}, we therefore obtain \pref{ipppositif} by letting $\eps\to 0^+$.

\end{proof}

\begin{thm}\label{theo:opticond}
Assume that $\alpha>1-2/d$ and that $\rho_0$, $\rho_1$ belong to $L^{\infty}(\T^d)$.  Let $(\rho, w)\in L^{\infty}\times L^2$ be a weak solution of \pref{cont}. Then $(\rho, w)$ solves \pref{dns} if and only if there exists $\phi \in \BV\cap L^{\infty}$ such that $\partial_t \phi \leq 0$, $\nabla \phi \in  L^2$ and
\[\left\{\begin{array}{llll}
& w=\frac{1}{2} \rho^\alpha \nabla \phi \mbox{ so that } \partial_t \rho +\dive\Big(\frac{1}{2} \rho^\alpha \nabla \phi\Big)=0,\\
&\rho>0 \Rightarrow \partial_t \phi^{\ac} +\frac{\alpha}{4} \rho^{\alpha-1} \vert \nabla \phi \vert^2=0, \\
&\rho=0 \Rightarrow w=\nabla \phi=0, \\
& \int_0^1 \int_Q \rho \partial_t \phi^{\ac}+w \cdot \nabla \phi =\int_Q \phi(1,.) \rho_1-\int_Q \phi(0,.) \rho_0.
\end{array}\right.\]
\end{thm}

\begin{rem} $\;$\\
1) Let us notice that the optimality system in Theorem~\ref{theo:opticond} admits a unique solution $(\rho,\phi)$ (with the property $\partial_t \phi \leq 0$, and up to a constant for $\phi$). Indeed, we have uniqueness in the dual problem~\eqref{primalr}, up to a constant, and the distance problem~\eqref{dns} admits itself an unique solution $(\rho,w)$. \\

\noindent 2) Corollary 5.18 of \cite{dolnazsav} states that the distance problem~\eqref{dns}  provides a (unique) constant speed geodesics. This last property entails that
$$
t\mapsto \int_Q\frac{|w|^2}{\rho^\alpha}dx
$$
is constant on $[0,1]$, and then, we can slightly improve the regularity properties of the potential $\phi$ by
\begin{equation*}
\left\{
\begin{array}{l}
\ds \nabla\phi \in L^{\infty}([0,1],L^2(\T^d))\\
\ds \partial_t\phi^{{\rm ac}} \in L^{\infty}([0,1],L^1(\T^d)).
\end{array}
\right.
\end{equation*}

\noindent 3) If $(\rho,w)\in L^{\infty}\times L^2$ solves \pref{dns} and if, in addition,  for every $t$, $\rho_t\geq C>0$ a.e. on $\T^d$ then the optimality conditions of Theorem \ref{theo:opticond} can be improved. In this case, $\partial_t \phi$ has no singular part. Indeed, proceeding as in the proof of lemma \ref{ippineg} we obtain
\[0\leq -C \Vert \partial_t \phi^s\Vert_{\M((0,1)\times \T^d)}+ \int_0^1 \int_Q \rho \partial_t \phi^{\ac}+w \cdot \nabla \phi +\int_Q \phi(0,.) \rho_0-\int_Q \phi(1,.) \rho_1\]
which together with the last condition of Theorem \ref{theo:opticond} gives the desired claim. In this case, we therefore have  $\phi\in W^{1,1}((0,1)\times \T^d)$ and since $\rho$ is bounded away from $0$, $\phi$ satisfies the Hamilton-Jacobi equation
 \[\partial_t \phi +\frac{\alpha}{4} \rho^{\alpha-1} \vert \nabla \phi \vert^2=0\]
 almost everywhere.  According to Theorem~\ref{carrillo}, in dimension $1$, if $\rho_0$, $\rho_1$ are bounded away from $0$ then so is $\rho_t$ (uniformly in $t$) so that the previous properties hold. Unfortunately, we do not know whether  the same holds in higher dimension.
\end{rem}

\begin{proof}[Proof of Theorem~\ref{theo:opticond}]
Let $(\rho, w)$ solve \pref{dns} (so that  $(\rho, w)\in L^{\infty}\times L^2$  by Theorem \ref{carrillo}). We deduce from Theorem \ref{existdualdns} that there is a unique (up to a constant)  solution $\phi\in K$ of \pref{primalr} and it satisfies $\nabla \phi\in L^2$. By the duality relation of proposition \ref{dualitydns}, the fact that $L$ and $H$ are convex conjugates and Young's inequality, we then have
\[\begin{split}
0&=\int_0^1\int_Q H(\rho, w)+L(\partial_t \phi^{\ac}, \nabla \phi)+\int_Q (\phi(0,.) \rho_0-\phi(1,.) \rho_1)dx\\
&\geq  \int_{0}^1 \int_Q (\rho \partial_t \phi^{\ac} + w \cdot \nabla \phi)dxdt +\int_Q (\phi(0,.) \rho_0-\phi(1,.) \rho_1)dx.
\end{split}\]
With \pref{ipppositif}, we then have
\[\int_0^1 \int_Q \rho \partial_t \phi^{\ac}+w \cdot \nabla \phi =\int_Q \phi(1,.) \rho_1-\int_Q \phi(0,.) \rho_0\]
so that we should also have a.e. an equality in Young's inequality above. This means that $(\partial_t \phi^{\ac}(t,x), \nabla \phi(t,x))\in \partial H(\rho(t,x), w(t,x))$ for a.e. $(t,x)\in (0,1)\times \T^d$. Therefore  $w=\frac{1}{2} \rho^\alpha \nabla \phi$ a.e., for a.e. $(t,x)$ for which $\rho(t,x)=0$ one has  $\nabla \phi(t,x)=w(t,x)=0$ and for a.e. $(t,x)$ for which $\rho(t,x)>0$ one has 
\[\partial_t \phi^{\ac}(t,x)=-\alpha\frac{\vert w(t,x)\vert^2}{\rho^{\alpha+1}(t,x)}= -\frac{\alpha}{4} \rho^{\alpha-1} \vert \nabla \phi(t,x) \vert^2.\]
This proves the necessity claim. To prove sufficiency, assume that $(\rho, w)$ and $\phi$ satisfy the claim of the theorem. Let $(\mu, v)$ solve the continuity equation \pref{cont} with $(\mu,  v)\in L^{\infty}\times L^2$ (which is without loss of generality in view of theorem \ref{carrillo}) and $\int_0^1 \int_Q H(\mu(t,x), v(t,x))dx dt<+\infty$. Since, as previously, $(\partial_t \phi^{\ac}(t,x), \nabla \phi(t,x))\in \partial H(\rho(t,x), w(t,x))$ for a.e. $(t,x)\in (0,1)\times \T^d$, we have 
\[\begin{split}
&\int_0^1 \int_Q H(\mu(t,x), v(t,x))dx dt- \int_0^1 \int_Q H(\rho(t,x), w(t,x))dx dt\\
&\geq \int_0^1 \int_Q [(\mu-\rho)\partial_t \phi^{\ac}+(v-w)\cdot \nabla \phi]dx dt\\
&=\int_0^1 \int_Q (\mu \partial_t \phi^{\ac}+v\cdot \nabla \phi)+\int_Q (\phi(0,x) \rho_0(x)- \phi(1,x) \rho_1(x))dx
\end{split}.\]
By Lemma \ref{ippineg}, we have
\[\int_{0}^1 \int_Q (\mu \partial_t \phi^{\ac} +v \cdot \nabla \phi)dxdt\geq \int_Q (\phi(1,.) \rho_1-\phi(0,.) \rho_0)dx\]
which finally enables us to conclude that $(\rho, w)$ solves \pref{dns}. 
\end{proof}

\section{A problem on measures on curves}\label{sec:PbDualCourbe}

Let $\Gamma$ be the set of continuous curves $\gamma:[0,1]\to \R^d$ endowed with the topology of the uniform convergence and $W$ be the subset of such curves which are absolutely continuous and have a derivative in $L^{2/(2-\alpha)}([0,1])$. For $t\in [0,1]$, we denote by $e_t$ the evaluation map at time $t$: $e_t(\gamma)=\gamma(t)$. We denote by $\Pi$ the set of Borel probability measures $\eta$ on $\Gamma$ which are $\Z^d$ periodic in space: if $\tau_k$ is the translation in $\R^d$ with vector $k\in \Z^d$ and $\tau_k(\gamma)(s)=\gamma(s)+k$, then $\tau_k\sharp \eta=\eta$. Let $\Pi_2$ be the subset of such measures with
$$
\int_{\Gamma}\int_0^1 |\dot \gamma(t)|^{\frac{2}{2-\alpha}}  dt d\eta(\gamma)<+\infty\;.
$$
To any measure $\eta\in \Pi_2$ we associate the measure $\sigma_\eta$ on $[0,1]\times Q$ defined by the equality
$$
\int_0^1 \int_{Q}f(s,x)d\sigma_\eta(s,x)= \int_\Gamma\int_0^1 f(s,\gamma(s))\left|\dot \gamma(s)\right|^{\frac{2}{2-\alpha}} dsd\eta(\gamma)\;,
$$
for any continuous and periodic in space map $f$. We denote by $\Pi_{2,ac}$ the set of measures $\eta\in \Pi_2$ such that the measure $\sigma_\eta$ is absolutely continuous with respect to the Lebesgue measure. In this case we identify $\sigma_\eta$ with its density. Finally, given two probability densities $\rho_0, \rho_1$, we denote by $\Pi_{2,ac}(\rho_0,\rho_1)$ the set of $\eta\in \Pi_{2,ac}$ such that $e_0\sharp \eta=\rho_0$ and   $e_1\sharp \eta=\rho_0$. We set
$$
K(\eta)= \left\{\begin{array}{ll}
\ds \int_0^1\int_{Q} \left(\sigma_\eta(s,x)\right)^{2-\alpha}
 dx ds & {\rm if }\; \eta\in \Pi_{2,ac}(\rho_0,\rho_1)\\
+\infty & {\rm otherwise}
\end{array}\right. 
$$
Let us note for later use that:
\begin{lem}\label{lem:Ksci}
The map $K$ is lower semicontinuous on $\Pi$. 
\end{lem}
\begin{proof} Let $\eta_n$ weakly converge to $\eta$. Without loss of generality we can assume that $M:=\limsup_{n\to+\infty} K(\eta_n)<+\infty$. Since $(\sigma_{\eta_n})$ is bounded in $L^{2-\alpha}$ we can assume without loss of generality that $(\sigma_{\eta_n})$ converges weakly in $L^{2-\alpha}$ to some $\sigma \in L^{2-\alpha}$. For any continuous periodic and positive map $f$ the map 
$$
\gamma \to \left\{\begin{array}{ll}
\ds  \int_0^1 f(s,\gamma(s))\left|\dot \gamma(s)\right|^{\frac{2}{2-\alpha}} ds &{\rm if } \gamma \in W\\
+\infty & {\rm otherwise}
\end{array}\right.
$$
is lower semicontinuous in $\Gamma$. Hence 
$$
\begin{array}{rl}
\ds \int_0^1 \int_Q f(s,x) d\sigma_\eta(s,x)\; = & \ds  \int_\Gamma\int_0^1 f(s,\gamma(s))\left|\dot \gamma(s)\right|^{\frac{2}{2-\alpha}}\ dsd\eta(\gamma)\\
\leq  & \ds \liminf_{n\to+\infty}  \int_\Gamma\int_0^1 f(s,\gamma(s))\left|\dot \gamma(s)\right|^{\frac{2}{2-\alpha}}\ dsd\eta_n(\gamma)\\
= & \ds \liminf_{n\to+\infty}   \int_0^1 \int_Q f(s,x) \sigma_{\eta_n}(s,x)\\
= & \ds  \int_0^1 \int_Q f(s,x) \sigma (s,x) 
\end{array}
$$
Therefore, $\sigma_\eta$ is absolutely continuous, with $\sigma\geq \sigma_\eta$ and 
$$
K(\eta) =  \int_0^1 \int_Q \sigma_\eta^{2-\alpha}\; \leq \;  
 \int_0^1 \int_Q \sigma^{2-\alpha} \; \leq \; \liminf_{n\to +\infty} K(\eta_n)\;.
$$
\end{proof}

\begin{thm}\label{prop:equicurve} Let $\rho_0$ and $\rho_1$ be two Borel probability measures on $\T^d$. Then 
\begin{equation}\label{pb:gamma}
\inf_{\eta\in \Pi_{2,ac}(\rho_0,\rho_1)} K(\eta) = 
W_\alpha(\rho_0, \rho_1)^2.
\end{equation}
In addition, this problem is attained by some $\overline\eta\in\Pi_{2,ac}(\rho_0,\rho_1)$.
\end{thm}

\begin{proof}  In  order to show that $\inf_{\eta\in \Pi_{2,ac}(\rho_0,\rho_1)} K(\eta) \geq 
W_\alpha(\rho_0, \rho_1)^2$, we recall that, from proposition \ref{dualitydns}
$$
W_\alpha(\rho_0, \rho_1)^2= -\inf_{\phi\in C^1([0,1]\times \T^d)} J(\phi) \;.
$$
Let  $\gamma\in W$ and $\phi$ be a ${\mathcal C}^1$ periodic map such that $\partial_t\phi<0$. We set
$$
A(t,x)= \frac{|\nabla \phi(t,x)|^{\frac{2}{\alpha}}}{(-\partial_t \phi(t,x))}
$$
and 
\begin{equation}\label{def:dalpha}
d_\alpha= \frac{(2-\alpha)\alpha^{\frac{\alpha}{2-\alpha}}}{2^{\frac{2}{2-\alpha}}}\;.
\end{equation}
Following  the same  computations as in the proof of Proposition \ref{prop:constrphin}, we have
$$
\frac{d}{dt} \left(\phi(t,\gamma(t)) -d_\alpha\int_0^t (A(s,\gamma(s)))^{\frac{\alpha}{2-\alpha}}\left|\dot \gamma(s)\right|^{\frac{2}{2-\alpha}} ds\right) \leq 0
$$
Therefore 
\begin{equation}\label{phi1phi0gamma}
\phi(1,\gamma(1))\leq \phi(0,\gamma(0))+d_\alpha\int_0^1 (A(s,\gamma(s))^{\frac{\alpha}{2-\alpha}}\left|\dot \gamma(s)\right|^{\frac{2}{2-\alpha}}ds\;.
\end{equation}
Let now $\eta\in \Pi_{2,ac}(\rho_0, \rho_1)$. Integrating (\ref{phi1phi0gamma}) with respect to $\eta$ then gives:
$$
\begin{array}{l}
\ds{ \int_{Q} \phi(1,x)d\rho_1(x)-\int_Q \phi(0,x)d\rho_0(x)    } \\
\qquad \qquad \qquad  \leq \;  
\ds{   d_\alpha \int_{\Gamma} \int_0^1 (A(s,\gamma(s))^{\frac{\alpha}{2-\alpha}}\left|\dot \gamma(s)\right|^{\frac{2}{2-\alpha}}\ dsd\eta(\gamma)  }\\
\qquad \qquad \qquad = \;  \ds{ d_\alpha \int_{Q} \int_0^1 (A(s,y))^{\frac{\alpha}{2-\alpha}}\sigma_\eta(s,y)\ dsdy  }
\end{array} 
$$
Therefore
$$
\begin{array}{l}
\ds{ \kappa_\alpha  \int_0^1\int_{Q} \left(\frac{|\nabla  \phi|^{\frac{2}{\alpha}}}{(-\partial_t \phi)}\right)^{\frac{\alpha}{1-\alpha}} dx dt 
+\int_Q \phi(0)d\rho_0 -\int_Q \phi(1)d\rho_1  }\\
\qquad \qquad \geq 
\ds{   \int_0^1\int_{Q}  \left(\kappa_\alpha \left(A(s,y)\right)^{\frac{\alpha}{1-\alpha}}-
d_\alpha  (A(s,y))^{\frac{\alpha}{2-\alpha}}\sigma_\eta(s,y)
\right)dxdt }\\
\qquad \qquad \ds{ \geq  -  \int_0^1\int_{Q} \left(\sigma_\eta(s,y)\right)^{2-\alpha}
 dx dt\;. }\\
\end{array}
$$
So we have proved that 
$$
-W_\alpha(\rho_0, \rho_1)^2= \inf_{\phi\in C^1([0,1]\times \T^d)} J(\phi)  \geq - \inf_{\eta\in \Pi_{2,ac}(\rho_0,\rho_1)} K(\eta)\;.
$$

We now show the opposite inequality. For this, let $(\rho,w)$ be optimal in the problem
$$
W_\alpha(\rho_0, \rho_1)^2=\inf_{(\rho, w)} \int_0^1 \int_Q H(\rho^{\ac}, w) dx dt
$$
Let  $(\xi_\eps)_{\eps>0}$ be the heat kernel on $\R^d$. Then following \cite{dolnazsav} the pair
$(\rho_\eps, w_\eps)= (\rho, w)\star \xi_\eps$ satisfies the continuity equation
$$
\partial_t \rho_\eps +\dive(w_\eps)=0
$$
and is such that $\rho^{\ac}_\eps=\rho_\eps$ and 
$$
\lim_{\eps\to 0} \int_0^1 \int_Q H(\rho^{\ac}_\eps, w_\eps) dx dt
=
\int_0^1 \int_Q H(\rho^{\ac}, w) dx dt= W_\alpha(\rho_0, \rho_1)^2\;.
$$
Note also that $\rho_\eps$ is periodic, smooth in space and bounded below by a positive constant, while $w_\eps$ is also periodic and smooth. Therefore the flow
\begin{equation}\label{eq:eqdiff}
\left\{\begin{array}{l}
\ds \frac{d}{dt} X^x_t= \frac{w_\eps(t, X^x_t)}{\rho_\eps(t, X^x_t)}\qquad t\in [0,1]\\
X^x_0=x
\end{array}\right.
\end{equation}
is well defined and satisfies $X^{x+k}_t=X^x_t$ for any $k\in \Z^d$. From standard properties related to the continuity equation (see \cite{dm} for instance), we have 
$$
\rho_\eps(t,x)dx = X^\cdot_t \sharp \rho_{\eps,0}\qquad {\rm where }\; \rho_{\eps,0}= \rho_0\star \xi_\eps (x)dx
$$
Let us defined $\eta_\eps\in \Pi$ by 
$$
\int_{\Gamma} f(\gamma)d\eta_\eps(\gamma)= \int_{Q} f(X^x_\cdot) d\rho_{\eps,0}(x)
$$
for any continuous periodic and bounded map $f$ on $\Gamma$. Let us compute $\sigma_{\eta_\eps}$. We have 
$$
\begin{array}{rl}
\ds \int_0^1\int_Q f(t,x) d\sigma_{\eta_\eps}(t,x)\; = & \ds \int_Q\int_0^1 f(t,X^x_t) \left| \frac{w_\eps(t, X^x_t)}{\rho_\eps(t, X^x_t)}\right|^{\frac{2}{2-\alpha}}
dt d\rho_{\eps,0}(x)\\
= & \ds \ds \int_Q\int_0^1 f(t,y) \left| \frac{w_\eps(t, y)}{\rho_\eps(t, y)}\right|^{\frac{2}{2-\alpha}}
\rho_{\eps}(t,y) dtdy
\end{array}
$$
for any continuous, periodic map $f$ on $[0,1]\times \R^d$. 
Therefore 
$$
\sigma_{\eta_\eps}(t,x)= \frac{|w_\eps(t, x)|^{\frac{2}{2-\alpha}}}{ (\rho_\eps(t,x))^{\frac{\alpha}{2-\alpha}}}\;,
$$
which shows that $\eta_\eps\in \Pi_{2,ac}$. Moreover
$$
K(\eta_\eps)
=
\int_0^1\int_{Q}  \frac{|w_\eps(t, x)|^{2}}{ (\rho_\eps(t,x))^{\alpha}}
dxdt
= \int_0^1 \int_Q H(\rho_\eps, w_\eps) dx dt
$$
In particular, as $\eps\to 0$, $K(\eta_\eps)\to  W_\alpha(\rho_0, \rho_1)^2$, which conclude the first part of the theorem. The existence of a minimizer is a straightforward consequence of the next Lemma---the proof of which is postponed---stating that there is a subsequence $\eta_{\eps_n}$ which  converges weakly to some $\eta\in \Pi_{2,ac}$, together with Lemma~\ref{lem:Ksci}.
\end{proof}

\begin{lem}\label{lem:etan} Let $(\eta_n)$ be a sequence in $\Pi_{2,ac}$ such that $K(\eta_n)\leq C$ for some constant $C$. Then, up to a subsequence, $(\eta_n)$ weakly converges to some $\eta\in \Pi_{2,ac}$. 
\end{lem}

\begin{proof}[Proof of Lemma \ref{lem:etan}] For $R>0$, let $E_R=\{\gamma \in  W\;, \; \|\dot \gamma\|_{\frac{2}{2-\alpha}}> R\}$, then 
$$
\begin{array}{rl}
\ds R^{\frac{2-\alpha}{2}} \eta_n (E_R) \; \leq  & \ds  \int_{E_R}\int_0^1 \left|\dot \gamma(t)\right|^{\frac{2}{2-\alpha}} dt  d\eta_n(\gamma)\\
\leq &  \ds \int_{\Gamma }\int_0^1 \left|\dot \gamma(t)\right|^{\frac{2}{2-\alpha}} dt  d\eta_n(\gamma)\\
\leq & \ds  \int_Q \int_0^1 \sigma_{\eta_n}(t,x)  dt dx\\
\leq & \ds \left[ \int_Q \int_0^1 (\sigma_{\eta_n}(t,x))^{2-\alpha}  dt dx\right]^{\frac{1}{2-\alpha}}\; \leq \; C^{\frac{1}{2-\alpha}}
\end{array}
$$
Since $W\backslash E_R$ is relatively compact in $\Gamma$,  the sequence $(\eta_n)$ is tight, which implies that it has a converging  subsequence. 
\end{proof}

\begin{rem} Assume $\eta\in \Pi_{2,ac}(\rho_0,\rho_1)$ is optimal for the problem (\ref{pb:gamma}). Then, for any $\eta'\in \in \Pi_{2,ac}(\rho_0,\rho_1)$, one has $\eta_\lambda :=(1-\lambda) \eta+\lambda \eta'\in  \Pi_{2,ac}(\rho_0,\rho_1)$ and $\sigma_{\eta_\lambda}= 
(1-\lambda) \sigma_\eta+\lambda \sigma_{\eta'}$. Therefore one gets as optimality condition for $\eta$ the inequality:
$$
\int_0^1\int_Q \sigma_\eta^{1-\alpha}(\sigma_{\eta'}-\sigma_\eta)  dx dt \; \geq \; 0\qquad \forall \eta'\in \in \Pi_{2,ac}(\rho_0,\rho_1)\;.
$$
Since the problem is convex, this necessary condition is also sufficient. 

Recalling the definition of $\sigma_\eta$, the above inequality heuristically means that any $\bar \gamma$ in the support of $\eta$ is optimal for the problem
$$
\inf_\gamma  \int_0^1  \sigma_\eta^{1-\alpha}(t, \gamma(t)) \left|\dot \gamma(t)\right|^{\frac{2}{2-\alpha}}  dt
$$
where the infimum is taken over the curves $\gamma$ such that $\gamma(0)=\bar \gamma(0)$ and $\gamma(1)=\bar \gamma(1)$.
\end{rem}

Next we explain some relations between the minimizer $(\rho, w)$ of \pref{dns}, the minimizer $\phi$ of \pref{primalr} and a minimizer $\eta$ of (\ref{pb:gamma}). 

\begin{prop}\label{prop:Liens} Assume that $\rho_0$ and $\rho_1$ belong to $L^\infty(\T^d)$. Let $(\rho,w)$, $\phi$ and $\eta$ be optimal for \pref{dns}, \pref{primalr} and \pref{pb:gamma} respectively. Then
$$
\sigma_\eta^{2-\alpha}= \frac{ |w|^2}{\rho^\alpha}= \frac{\kappa_\alpha}{1-\alpha} \frac{|\nabla \phi|^{\frac{2}{1-\alpha}}}{(-\partial_t\phi^{\ac})^{\frac{\alpha}{1-\alpha}}}
\qquad {\rm a.e.}
$$
with the convention that $0/0=0$ and $a/0=+\infty$ if $a>0$. Moreover, there is a minimizer $\eta$ of (\ref{pb:gamma}) such that $\rho_t= e_t\sharp \eta$ and
\begin{equation}\label{LienEtaRhoW}
\int_0^1\int_Q \langle F(t,x), w(t,x)\rangle  dx  dt =  \int_{\Gamma}\int_0^1 \langle F(t,\gamma(t)), \dot \gamma(t)\rangle   dt  d\eta(\gamma)
\end{equation}
for any continuous, periodic map $F:[0,1]\times \R^d\to \R^d$. 
\end{prop}

\begin{proof} The second equality is a straightforward consequence of Theorem~\ref{theo:opticond}. As for the first one, let us use the strategy of proof of Theorem~\ref{prop:equicurve}. Let  $(\xi_\eps)_{\eps>0}$ be the heat kernel on $\R^d$, set $(\rho_\eps,w_\eps)= (\rho,w)\star \xi_\eps$,  and consider $\eta_\eps\in \Pi$ defined by 
$$
\int_{\Gamma} f(\gamma)d\eta_\eps(\gamma)= \int_{Q} f(X^x_\cdot) d\rho_{\eps,0}(x)
$$
for any continuous periodic and bounded map $f$ on $\Gamma$, where $X^x_\cdot$ is the solution of the differential equation (\ref{eq:eqdiff}). 
We already know that
$$
\sigma_{\eta_\eps}= \frac{|w_\eps(t, x)|^{\frac{2}{2-\alpha}}}{\rho_\eps^{\frac{\alpha}{2-\alpha}}}
$$
and that, up to some sequence,  
$\eta_\eps$ converges to some $\bar \eta$ which is optimal for \pref{pb:gamma}.  Since $(\sigma_{\eta_\eps})$ is bounded in $L^{2-\alpha}$, we can consider a weak limit $\sigma$ of the $(\sigma_{\eta_\eps})$ in $L^{2-\alpha}$ and we have explained in the proof of Lemma \ref{lem:Ksci} that $\sigma\geq \sigma_{\bar \eta}$. Hence
$$
W_\alpha(\rho_0,\rho_1)^2=\lim_{\eps\to 0} K(\eta_\eps) \geq \int_0^1\int_Q \sigma^{2-\alpha}\geq \int_0^1\int_Q \sigma_{\bar \eta}^{2-\alpha}\geq 
W_\alpha(\rho_0,\rho_1)^2.
$$

This implies that $\sigma_{\eta_\eps}$ strongly converges in $L^{2-\alpha}$ to $\sigma_{\bar \eta}$. The map $\eta\to \sigma_{\eta}$ being affine and $t\to t^{2-\alpha}$ strictly convex, we must have $\sigma_\eta=\sigma_{\bar \eta}$. Furthermore, 
$$
\sigma_{\eta_\eps}^{2-\alpha}=\frac{ |w_\eps(t, x)|^{2}}{\rho_\eps^{\alpha}} \to \frac{|w|^2}{\rho^\alpha} \qquad \mbox{\rm  a.e. on } \{\rho>0\}\;{\rm as}
\; \eps\to 0\;,
$$
so that $\sigma_\eta^{2-\alpha}= \frac{|w|^2}{\rho^\alpha}$  a.e. on $\{\rho>0\}$. To complete the proof we note that 
$$
\begin{array}{rl}
W_\alpha(\rho_0,\rho_1)\; = & \ds  \int_0^1\int_Q \sigma_\eta^{2-\alpha}\; \geq \;  \int_0^1\int_Q \sigma_\eta^{2-\alpha}{\bf 1}_{\{\rho>0\}} \\
= &  \ds  \int_0^1\int_Q \frac{|w|^2}{\rho^\alpha}{\bf 1}_{\{\rho>0\}} \; 
= \;  \int_0^1\int_Q \frac{|w|^2}{\rho^\alpha} \; 
= \; W_\alpha(\rho_0,\rho_1)
\end{array}
$$
Therefore $\sigma_\eta=0=\frac{|w|^2}{\rho^\alpha}$  a.e.  on $\{\rho=0\}$. 

Let us finally check that $\eta$ satisfies \pref{LienEtaRhoW}. By definition of $\eta_\eps$, $e_t\sharp \eta_\epsilon= \rho_{\eps,t}$. Moreover, for any continuous, periodic map $F:[0,1]\times \R^d\to \R^d$, we have 
$$
 \int_0^1\int_Q \langle F(t,x), w_\eps(t,x)\rangle  dx  dt =  \int_{\Gamma}\int_0^1 \langle F(t,\gamma(t)), \dot \gamma(t)\rangle   dt  d\eta_\eps(\gamma).
 $$
 Defining $I_F(\gamma):=  \int_0^1 \langle F(t,\gamma(t)), \dot \gamma(t)\rangle dt$, it is therefore enough to prove that $\int_\Gamma I_F d\eta_\eps \to \int_\Gamma I_F d\eta$ 
 as $\eps\to 0$ to obtain  \pref{LienEtaRhoW}. Let $R>0$ and
$$
B_R:=\{\gamma\in W^{1,\frac{2}{2-\alpha}}, \; \gamma(0)\in Q, \;  \|\dot \gamma\|_{\frac{2}{2-\alpha}}\leq R\},
$$
for $\N\in \N^*$ and $\gamma\in W^{1, \frac{2}{2-\alpha}}$, define 
  \[I_F^N(\gamma):=\sum_{k=0}^{N-1} \<F(k/N, \gamma(k/N), \gamma((k+1)/N)-\gamma(k/N)>.\]
 By standard uniform continuity arguments, observe that for fixed $R>0$, $\delta_{R,N}:=\sup_{\gamma \in B_R} \vert I_F(\gamma)-I_F^N(\gamma)\vert$ tends to $0$ as $N\to \infty$. Now since $F$ is bounded and $R^{\frac{2-\alpha}{2}}(\eta_\eps+\eta)(\Gamma\setminus B_R)\leq C$ (see Lemma~\ref{lem:etan})  we have
\[\begin{split}
\Big \vert \int_\Gamma I_F d(\eta_\eps-\eta)\Big \vert \leq C \frac{\Vert F \Vert_{\infty}}{R^{\frac{2-\alpha}{2}}}+  \Big \vert \int_{\Gamma\cap B_R} I_F d(\eta_\eps-\eta)  \Big \vert\\
 \leq C \frac{\Vert F \Vert_{\infty}}{R^{\frac{2-\alpha}{2}}}+ \Big \vert \int_{\Gamma\cap B_R} I_F^N d(\eta_\eps-\eta) \Big \vert+2 \delta_{R,N}.\end{split}\]
Finally, since $I_F^N$ is continuous for the $C^0$ topology, $\int_\Gamma I_F^N d\eta_\eps \to \int_\Gamma I_F^N d\eta$  as $\eps\to 0$ which is enough to conclude. 
 
\end{proof}

\end{document}